\documentclass[a4paper,fleqn]{cas-sc}

\usepackage[numbers]{natbib}
\def\tsc#1{\csdef{#1}{\textsc{\lowercase{#1}}\xspace}}
\tsc{WGM}
\tsc{QE}
\newtheorem{theorem}{Theorem}
\newtheorem{lemma}[theorem]{Lemma}
\newtheorem{proposition}{Proposition}

\newproof{pf}{Proof}
\newproof{pot}{Proof of Theorem \ref{thm}}
\usepackage{bm}
\usepackage{subcaption}
\usepackage{tikz}
\usepackage{graphicx}
\usepackage{float} 
\usepackage{algorithm}
\usepackage{algorithmic}

\numberwithin{equation}{section}

\newcommand{\Lrho}{\mathcal{L}_{\rho}}

\newcommand{\kap}{\kappa}
\newcommand{\wlim}{\rightharpoonup}
\newcommand{\slim}{\rightarrow}
\newcommand{\Gf}{\Gamma_f}

\newcommand{\Ltwo}{L^2}
\newcommand{\normLtwo}[1]{\left\|#1\right\|_{L^2(\Gf)}}
\newcommand{\normV}[1]{\left\|#1\right\|_{V}}
\newcommand{\innerGf}[2]{\langle #1,\, #2 \rangle_{\Gf}}

\newcommand{\Argmin}{\operatorname{Arg\,min}}

\begin{document}
	\def\floatpagepagefraction{1}
	\def\textpagefraction{.001}
	
	\shorttitle{A Dual ADMM Framework for Inverse  Parameter Identification in Elliptic Variational Inequalities}    
	
	\shortauthors{Y.Mezzan, Y. Ouakrim, A. Zafrar}  
	
	\title [ ]{A Dual ADMM Framework for Inverse Parameter Identification in Elliptic Variational Inequalities}  
	
	\tnotetext[1]{} 
	
	\author{Youness Mezzan$^{1}$}
	\ead{youness.mezzan@ced.uca.ma}
	\address{$1$ Department of Mathematics, Cadi Ayyad University, Faculty of Science and Technology, Marrakech, Morocco.}
	
	\author{Youssef Ouakrim$^{2}$}
	\ead{youssef.ouakrim@usmba.ac.ma}
	\address{$2$ Laboratoire de Math\'ematiques et Applications aux Sciences de l'Ing\'enieur, Ecole Normale Sup\'erieure, Universit\'e Sidi Mohamed Ben Abdellah, F\`es, Maroc.}
	
	\author{Abderrahim Zafrar$^{3,*}$}
	\ead{ zafrar.abd@gmail.com}
	\address{$3$ Department of Mathematics, Sidi Mohamed Ben Abdellah University, Faculty of Sciences Dhar El Mahraz,  Fez, Morocco.}
	
\begin{abstract}
	This paper is concerned with an inverse parameter identification problem governed by a second-kind elliptic variational inequality involving a non-differentiable boundary functional. The objective is to reconstruct an unknown Robin-type coefficient from boundary measurements while preserving the intrinsic nonsmooth structure of the underlying model. The inverse problem is formulated as a constrained optimization problem, where the cost functional measures the discrepancy between the computed and observed data and is stabilized by Tikhonov regularization. To overcome the difficulties associated with the nonsmooth variational inequality, we reformulate the state problem as a constrained minimization problem and employ an augmented Lagrangian framework combined with the alternating direction method of multipliers (ADMM). Fenchel duality is used to treat the non-differentiable boundary term without introducing smoothing approximations. We establish identifiability of the unknown coefficient and prove existence and convergence results for the proposed approximation scheme. Numerical experiments performed on a classical benchmark problem confirm the theoretical analysis and demonstrate that accurate parameter reconstruction can be achieved with reduced computational effort. The proposed methodology provides a robust framework for inverse problems arising in applications such as heat transfer, frictional contact, and corrosion detection.
\end{abstract}	

\maketitle	
\section{Introduction} \label{introduction}

Inverse parameter identification problems governed by variational inequalities arise in a broad spectrum of physical and engineering applications, including frictional contact mechanics, corrosion detection from boundary voltage measurements, and Robin-type heat
exchange modelling \cite{inglese1997inverse,kaup1996method,martin1998inverse,chantasiriwan2000inverse}. In each of these settings, the parameter of interest typically a friction or Robin coefficient defined on a portion of the boundary is not directly observable, but rather must be inferred
from boundary measurements of the associated state variable. The mathematical structure of such problems is considerably more delicate than its classical PDE counterpart: the state is governed by an elliptic variational inequality of the second kind
\cite{glowinski2008lectures,hlavacek2012solution,kinderlehrer2000introduction}, which involves a convex but non-differentiable functional encoding the friction or contact condition.  This non-smooth structure, combined with the presence of inequality type constraints in the state problem \cite{duvaut1976inequations}, renders the application of standard based identification methods largely intractable.

The theoretical foundations of elliptic variational inequalities of the second kind were laid in the fundamental works of Duvaut and Lions~\cite{duvaut1976inequations} and Glowinski et~al.~\cite{glowinski2008lectures}.  While their existence and uniqueness theory is by now well-understood, the numerical identification of unknown parameters within such inequalities has remained comparatively unexplored.  The principal difficulties are threefold.  First, the map $g \mapsto u_g$ from the parameter to the state is nonlinear, Lipschitz continuous, but generally neither smooth nor convex \cite{gwinner2018two}.  Second, the unknown coefficient is defined on the contact zone, which may itself depend on the solution and is not known a priori.  Third, the uniqueness of the forward state depends on the coefficient being below a mesh-dependent critical threshold, which further complicates the numerical treatment.

Several strategies have been proposed in the literature to address related classes of inverse problems.  \cite{gwinner2018two} and \cite{gwinner2018optimization} study the identification of parameters linked to bilinear and nonlinear non-smooth functionals in second-kind variational inequalities, establishing theoretical optimality conditions under Lipschitz assumptions. \cite{migorski2019inverse} address a nonlinear quasi-variational setting in Banach spaces, using fixed-point arguments to characterize the solution set.  The identification of a distributed parameter in an elliptic variational inequality is investigated in \cite{hintermuller2001inverse}, while \cite{hintermuller2008active} develops an output least-squares formulation solved via an active-set Newton method with feasibility restoration to circumvent the low regularity of the associated Lagrange multipliers. \cite{cen2022inverse} introduces a nonsmooth regularization framework encompassing total variation regularization, thereby enabling the identification of discontinuous coefficients.  More recently, \cite{zeng2023inverse} addresses parameter identification in nonlinear mixed quasi-variational inequalities within an abstract
Banach space framework.

A widely used alternative consists in replacing the non-smooth inequality constraint by a regularized variational equation \cite{facchinei2003finite,gwinner2017optimization}, yielding a smoother optimization landscape.  Despite its numerical convenience, this smoothing approach introduces an additional regularization bias and may compromise the physical interpretation of the underlying model, particularly in applications where the stick-slip transition or the corrosion threshold has intrinsic physical significance.

The present paper proposes a fundamentally different strategy: rather than regularizing the non-smooth structure of the state problem, we preserve it exactly and exploit convex duality to decompose the resulting constrained minimization problem into a sequence of tractable subproblems.  The core of our approach rests on three ingredients.
\begin{itemize}
	\item Augmented Lagrangian decomposition.  We reformulate the non-smooth variational inequality as an equivalent constrained minimization problem by introducing an auxiliary boundary variable $\phi$ subject to the coupling constraint $u|_{\Gamma_f} = \phi$.
	The augmented Lagrangian associated with this reformulation is then minimized by the Alternating Direction Method of Multipliers (ADMM) \cite{boyd2011distributed}, which alternates between a quadratic elliptic subproblem for $u$ and a non-smooth boundary
	subproblem for $\phi$.
	
	\item Fenchel duality for the non-differentiable term.  The $\phi$-subproblem involves the minimization of the friction functional	$j(g,\cdot) = \int_{\Gamma_f} g|\cdot|\,ds$, which is convex but non-differentiable.  By applying Fenchel duality \cite{ekeland1999convex}, the solution is obtained in closed form via a pointwise soft-thresholding (proximal projection) operator, entirely avoiding any smoothing approximation of the absolute value function.
	
	\item Approximate identification sequence.  The inverse problem  is approximated by a sequence of	regularized problems, indexed by the ADMM iteration count $n$.  At level $n$, the cost functional $J^n_\varepsilon(g)$ is constructed using the approximate
	state $u^n_g$ obtained after exactly $n$ ADMM iterations, rather than the exact solution of the forward inequality.  We prove that the resulting sequence of minimizers $g^{n_k}_*$ converges strongly in $L^2(\Gamma_f)$ to the true parameter $g^*$, and that
	the corresponding states converge strongly in $V$.  This result eliminates the computational bottleneck of solving the forward variational inequality to full accuracy at every optimization step.
\end{itemize}
The proposed framework has several notable advantages over existing approaches.  Unlike smoothing-based methods, it preserves the intrinsic non-smooth structure of the friction model and avoids the introduction of any additional smoothing parameter.  Unlike semismooth Newton methods \cite{stadler2004semismooth}, it does not require the computation of generalized derivatives of the state map, which are known to be of low regularity in the bilateral constraint setting.  The ADMM splitting reduces the original non-smooth constrained identification problem to a sequence of linear elliptic solves and pointwise threshold evaluations, all of which are computationally inexpensive.  Furthermore, the same augmented Lagrangian framework extends naturally to other coefficient identification tasks, including the determination of Lamé parameters in linear elasticity, Young moduli in structural mechanics, or spatially varying Robin coefficients in heat transfer \cite{lin2005linear,chantasiriwan2000inverse,martin1998inverse}.  It also applies
directly to bilevel optimization problems with variational inequality state constraints \cite{capatina2000optimal,essoufi2021optimal}.

The remainder of the paper is organized as follows. Section~\ref{Section2} introduces the functional setting, formulates the second-kind elliptic variational inequality and its equivalent energy minimization, and casts the inverse problem as a Tikhonov-regularized output least-squares program over the admissible set. Section~\ref{Section3} establishes the identifiability of the Robin coefficient. Section~\ref{Section4} constitutes the theoretical core of the paper. We derive the augmented Lagrangian reformulation of the state problem,
obtain closed-form ADMM subproblem solutions via Fenchel duality, define the approximate inverse problem sequence $\{(IP)_\varepsilon^n\}$, and prove existence of minimizers and their
strong convergence in $L^2(\Gamma_f)$ to the exact regularized solution.
Section~\ref{Section5} provides a quantitative numerical validation. Section~\ref{Section6} draws conclusions and outlines extensions to spatially varying coefficients, noisy data, and quasi-variational settings.

%%%%%%%%%%%%%%%%%%%%%%%%%%%%%%%%%%%%%%%%%%%%%%%%%%%%%%%%%%%%%%%%%%%%%%%%%%%
\section{Forward and Inverse Problem Formulation}\label{Section2}
Let $\Omega \subset \mathbb{R}^n$, $n \in \{2, 3\}$, be an open bounded domain
with Lipschitz boundary $\Gamma$. Let $\Gamma_0 \subset \Gamma$ be a relatively
open subset, possibly empty, and set $\Gamma_f := \Gamma \setminus \Gamma_0$.
We introduce the closed subspace
\begin{equation}
	V := \left\{ v \in H^1(\Omega) : \tau v = 0
	\;\text{ a.e.\ on } \Gamma_0 \right\},
\end{equation}
where $H^1(\Omega)$ denotes the standard Sobolev space and $\tau : H^1(\Omega) \to L^2(\Gamma)$ is the Sobolev trace operator. For brevity, we write $v$ in place of $\tau v$ when the trace is evaluated on $\Gamma$.

\subsection{The Direct Problem}

We consider the following second-kind elliptic variational inequality.

\noindent\textbf{Problem $(PV)$.}
\textit{Find $u \in V$ such that}
\begin{equation}
	\label{eq:VI}
	a(u,\, v - u) + j(g,\, v) - j(g,\, u)
	\geq \langle f,\, v - u \rangle
	\qquad \forall\, v \in V,
\end{equation}
\textit{where:}
\begin{itemize}
	\item $\displaystyle a(u, v) :=
	\int_{\Omega}\bigl(\nabla u \cdot \nabla v + u\,v\bigr)
	\,\mathrm{d}x$
	is a symmetric, continuous, and coercive bilinear form on
	$V \times V$\,;
	\item $\displaystyle j(g, v) :=
	\int_{\Gamma_f} g\,|v|\,\mathrm{d}s$
	is a proper, convex, and lower semi-continuous functional,
	with $g > 0$ a.e.\ on $\Gamma_f$\,;
	\item $\displaystyle \langle f, v \rangle :=
	\int_{\Omega} f\,v\,\mathrm{d}x$
	is a bounded linear functional, with $f \in L^2(\Omega)$.
\end{itemize}
The existence and uniqueness of a solution to Problem~$(PV)$ follow from standard results in the theory of variational inequalities \cite{kinderlehrer2000introduction, glowinski2008lectures,hlavacek2012solution}. The principal analytical difficulty lies in the non-differentiability of the functional $j(g, \cdot)$ with respect to its second argument. Two classical strategies are available to address this: the introduction of a Lagrange multiplier to resolve the subdifferential of the absolute value, or the regularization of $j(g, \cdot)$ by a smooth approximation. In the present work, we pursue the former approach within a dual augmented Lagrangian framework, as detailed in Section~\ref{Section4}.

\subsection{The Inverse Problem}
The objective of this paper is to recover the Robin coefficient $g$, which characterizes the material response on the contact boundary $\Gamma_f$, from boundary measurements of the state $u$. To ensure the well-posedness of this identification problem, we impose the following condition on the solution:
\begin{equation}
	\label{eq:meas_condition}
	\mathrm{meas}\!\left(
	\bigl\{ x \in \Gamma_f : u|_{\Gamma_f}(x) = 0 \bigr\}
	\right) = 0,
\end{equation}
where $\mathrm{meas}(W)$ denotes the Lebesgue measure of a set $W$. Assumption~\eqref{eq:meas_condition} is standard in the Robin inverse problem literature \cite{martin1998inverse, chantasiriwan2000inverse,lin2005linear}. It ensures that Newton's law of convective heat transfer is meaningful on $\Gamma_f$, and is equivalent to the injectivity of the parameter-to-state map $g \mapsto u_g$, a property that is naturally
satisfied in electrochemical corrosion detection applications~\cite{inglese1997inverse}.

Let $\Gamma_s \subseteq \Gamma_f$ denote the portion of the friction boundary on which measurements are available, and let $u_d \in L^2(\Gamma_s)$ be the corresponding observed data. We define the set of admissible coefficients as
\begin{equation}
	\label{eq:admissible_set}
	F := \left\{ g \in H^1(\Gamma_f) :
	\|g\|_{L^2(\Gamma_f)} \leq g_0 \right\},
\end{equation}
for a given bound $g_0 > 0$. The inverse problem then reads: find $g \in F$ such that the solution $u$ of Problem~$(PV)$ satisfies $u|_{\Gamma_s} = u_d$.\\

The mathematical formulation of this identification problem is cast as the following output least-squares minimization.\\

\noindent\textbf{Problem $(Ip)$.}
\textit{Find $g^* \in F$ such that}
\begin{equation}
	\label{eq:OCP}
	J(g^*) = \inf_{g \in F}\, J(g),
	\qquad
	J(g) := \frac{1}{2}
	\int_{\Gamma_f} \bigl(u_g - u_d\bigr)^2 \mathrm{d}s,
\end{equation}
where $u_g := u(g) \in V$ denotes the solution of Problem~$(PV)$ associated with the coefficient $g$.\\

Since the parameter-to-state mapping $g \mapsto u_g$ is nonlinear, albeit Lipschitz continuous \cite{gwinner2018optimization}, Problem~$(Ip)$ is generally ill-posed in the sense of Hadamard. To restore stability and guarantee the existence of minimizers, we employ Tikhonov regularization and introduce the penalized cost functional
\begin{equation}
	\label{eq:Jeps}
	J_\varepsilon(g) :=
	\frac{1}{2}\int_{\Gamma_f}\bigl(u_g - u_d\bigr)^2 \mathrm{d}s
	+ \frac{\varepsilon}{2}\,\|g\|^2_{L^2(\Gamma_f)},
	\qquad \varepsilon > 0,
\end{equation}
where the second term penalizes the $L^2(\Gamma_f)$-norm of $g$ and $\varepsilon$ is the regularization parameter. The regularized inverse problem is formulated as follows.\\

\noindent\textbf{Problem $(Ip)_\varepsilon$.}
\textit{Find $g^* \in F$ such that}
\begin{equation}
	\label{eq:OCPalp}
	J_\varepsilon(g^*) = \inf_{g \in F}\, J_\varepsilon(g).
\end{equation}
The existence of at least one solution to Problem~$(Ip)_\varepsilon$ follows from the coercivity and weak lower semi-continuity of $J_\varepsilon$ on the weakly closed set $F \subset L^2(\Gamma_f)$, as established in Section~\ref{Section4}.

\section{Identifiability of the Robin Coefficient}\label{Section3}
In this section, we investigate the identifiability of the Robin coefficient $g$ from boundary measurements. We recall that the state problem \eqref{eq:VI} admits a unique solution, which is formally equivalent to the following boundary value problem:
\begin{equation}\label{p1}
	\left\{\begin{array}{ll}
		-\Delta u + u = f & \text{in } \Omega, \\
		u = 0 & \text{on } \Gamma_0, \\
		\left|\frac{\partial u}{\partial \nu}\right| \leq g, \quad u \frac{\partial u}{\partial \nu} + g|u| = 0 & \text{on } \Gamma_f,
	\end{array}\right.
\end{equation}
where $\nu$ denotes the outward unit normal vector on $\partial\Omega$. The non-smooth boundary conditions on $\Gamma_f$ can be equivalently decomposed as follows:
\begin{equation}\label{ineq}
	\begin{aligned}
		\left|\frac{\partial u}{\partial \nu}\right| < g &\Longrightarrow u = 0, \\
		\frac{\partial u}{\partial \nu} = g &\Longrightarrow u \leq 0, \\
		\frac{\partial u}{\partial \nu} = -g &\Longrightarrow u \geq 0.
	\end{aligned}
\end{equation}
The following result establishes the identifiability of the Robin coefficient from boundary measurements.
\begin{proposition}\label{identifiability}
	Let $g_1, g_2 \in F$ be two parameters, and let $u_1, u_2$ be the corresponding solutions of problem  \eqref{eq:VI} associated with $g_1$ and $g_2$, respectively.  If $u_{1} = u_{2}$ on $\Gamma_f$, then $g_1 = g_2$ a.e. on $\Gamma_f$.
\end{proposition}
\begin{pf}
Let $g_1$ and $g_2$ be two elements of $F$ such that $u_{1\left.\right|_{\Gamma_f}}=u_{2\left.\right|_{\Gamma_f}}$. Let us define $w = u_1 - u_2$. By linearity, $w$ solves the following homogeneous boundary value problem:
\begin{equation}\label{p_w}
		\left\{\begin{aligned}
			-\Delta w+w &= 0 && \text{in } \Omega, \\
			w &= 0 && \text{on } \Gamma_0, \\
			w &= 0 && \text{on } \Gamma_f.
		\end{aligned}\right.
\end{equation}
By invoking the unique continuation property for this elliptic system, we immediately obtain that $w \equiv 0$ in $\Omega$. This implies that $u_1 = u_2$ throughout the domain $\Omega$, from which it follows that their normal derivatives also coincide on the boundary.
Consequently, evaluating the boundary conditions on $\Gamma_f$ for both states yields:
	$$
	\begin{aligned}
		u_1 \frac{\partial u_1}{\partial \nu} + g_1 |u_1| = 0 \quad \text{on } \Gamma_f, \\
		u_1 \frac{\partial u_1}{\partial \nu} + g_2 |u_1| = 0 \quad \text{on } \Gamma_f.
	\end{aligned}
	$$
Subtracting these two identities leads to:
	\begin{equation}\label{8}
		|u_1|\left(g_1-g_2\right)=0 \quad \text{on } \Gamma_f.
	\end{equation}
	Symmetrically, we also have:
	\begin{equation}\label{9}
		|u_2|\left(g_1-g_2\right)=0 \quad \text{on } \Gamma_f.
	\end{equation}
To establish the result, let us assume by contradiction that $g_1 \not \equiv g_2$. Taking advantage of the continuity of $g_1$ and $g_2$, there exists an open subset $\vartheta \subset \Gamma_f$ with strictly positive Lebesgue measure ($\text{meas}(\vartheta) > 0$) such that:
	$$
	\left(g_1-g_2\right)(x) \neq 0, \quad \forall x \in \vartheta.
	$$
In view of equations \eqref{8} and \eqref{9}, this condition implies that $u_i \equiv 0$ on $\vartheta$ for $i \in \{1,2\}$. Therefore, each $u_i$ turns out to be a solution of the modified problem:
	$$
	\left\{\begin{aligned}
		-\Delta u_i+u_i &= f && \text{in } \Omega, \\
		u_i &= 0 && \text{on } \Gamma_0, \\
		u_i &= 0 && \text{on } \vartheta, \\
		\left|\frac{\partial u_i}{\partial \nu}\right| \leq g_i, \quad u_i \frac{\partial u_i}{\partial \nu}+g_i|u_i| &= 0 && \text{on } \Gamma_f \setminus \vartheta.
	\end{aligned}\right.
	$$
Applying the same argument iteratively, we can extend this trivial trace to the whole boundary, showing that $u_i \equiv 0$ on $\Gamma_f$ for $i \in \{1,2\}$. Thus, $u_i$ satisfies both problem \eqref{p_w} and the standard Dirichlet problem:
	$$
	\left\{\begin{aligned}
		-\Delta u_i+u_i &= f && \text{in } \Omega, \\
		u_i &= 0 && \text{on } \Gamma,
	\end{aligned}\right.
	$$
where $\Gamma = \Gamma_0 \cup \Gamma_f$. This directly contradicts the assumption:
	$$
	\text{meas}\left(\{x \in \Gamma_f \;:\; u_{\left.\right|_{\Gamma_f}}(x)=0\}\right) = 0.
	$$
As a consequence:
	$$
	\left(g_1-g_2\right)(x) = 0, \quad \forall x \in \Gamma_f,
	$$
which completes the proof.
\end{pf}
\section{Augmented Lagrangian Decomposition, Algorithm, and Convergence Analysis}\label{Section4}
This section constitutes the theoretical and algorithmic core of the paper. We first reformulate the non-smooth state problem~\eqref{eq:VI} as an equivalent constrained minimization problem amenable to operator splitting. We then derive the ADMM subproblems and their closed-form solutions via Fenchel duality, define the approximate inverse problem sequence, and establish existence and strong convergence of the resulting minimizers.
\subsection{Augmented Lagrangian Formulation of the State Problem} \label{subsec:auglag}
Let $\mathcal{J} : V \to \mathbb{R}$ denote the energy functional
\begin{equation}
	\mathcal{J}(v) \;:=\; \tfrac{1}{2}\,a(v,v) \;-\; \langle f, v\rangle,
	\label{eq:energy}
\end{equation}
which is strictly convex and weakly coercive on $V$ by the assumptions on $a(\cdot,\cdot)$. It is classical (see, e.g.,~\cite{ekeland1999convex,glowinski2008lectures}) that the variational inequality~\eqref{eq:VI} is equivalent to the unconstrained minimization
\begin{equation}
	\min_{v \in V}\;\bigl[\mathcal{J}(v) + j(g, v)\bigr].
	\label{eq:min_unconstrained}
\end{equation}
In order to decouple the smooth part $\mathcal{J}$ from the non-differentiable functional $j(g,\cdot)$, we introduce an auxiliary variable $\phi \in \Ltwo(\Gf)$ representing the boundary trace of the state on $\Gf$, subject to the coupling constraint $u|_{\Gf} = \phi$. Problem~\eqref{eq:min_unconstrained} is then equivalently rewritten as the following constrained minimization problem.
\begin{equation}
	\text{Find}\;(u,\phi)\in V\times \Ltwo(\Gf)\;\text{such that}\quad
	\mathcal{J}(u)+j(g,\phi)\;\leq\;\mathcal{J}(v)+j(g,\psi)
	\quad\forall\,(v,\psi)\in V\times \Ltwo(\Gf),
	\label{eq:constrained_min}
\end{equation}
subject to the constraint $u - \phi = 0$ on $\Gf$.\\
The augmented Lagrangian associated with problem~\eqref{eq:constrained_min}
is defined, for a penalty parameter $\rho > 0$ and a Lagrange multiplier
$\lambda \in \Ltwo(\Gf)$, by
\begin{equation}
	\Lrho(u, \phi, \lambda)
	\;:=\;
	\mathcal{J}(u)
	\;+\; j(g, \phi)
	\;+\; \innerGf{\lambda}{u - \phi}
	\;+\; \frac{\rho}{2}\,\|u - \phi\|_{\Gf}^{2},
	\label{eq:aug_lag}
\end{equation}
where $\|\cdot\|_{\Gf}$ and $\innerGf{\cdot}{\cdot}$ denote the $\Ltwo(\Gf)$-norm
and inner product, respectively. The augmented Lagrangian~\eqref{eq:aug_lag}
coincides with the standard Lagrangian at feasible points, while the quadratic
penalty term $\frac{\rho}{2}\|u-\phi\|_{\Gf}^2$ enforces constraint satisfaction
progressively and improves the conditioning of the subproblems.

\subsection{ADMM Subproblems and Their Explicit Solutions} \label{subsec:subproblems}
Starting from an initial pair $(\phi^0, \lambda^0) \in \Ltwo(\Gf) \times \Ltwo(\Gf)$,
the ADMM produces a sequence $\{(u^n, \phi^n, \lambda^n)\}_{n \geq 1}$ by alternately minimizing $\Lrho$ with respect to $u$ and $\phi$, followed by a dual ascent update:
\begin{align}
	u^{n+1}   &\;\in\; \Argmin_{v \in V}\;
	\Lrho(v,\, \phi^n,\, \lambda^n),
	\label{eq:u_subproblem}\\[4pt]
	\phi^{n+1} &\;\in\; \Argmin_{\psi \in \Ltwo(\Gf)}\;
	\Lrho(u^{n+1},\, \psi,\, \lambda^n),
	\label{eq:phi_subproblem}\\[4pt]
	\lambda^{n+1} &\;=\; \lambda^n + \rho\,(u^{n+1} - \phi^{n+1}).
	\label{eq:multiplier_update}
\end{align}

Solution of the $u$-subproblem~\eqref{eq:u_subproblem}. Expanding $\Lrho(\cdot, \phi^n, \lambda^n)$ and completing the square, one verifies that~\eqref{eq:u_subproblem} is a strictly convex quadratic minimization problem over $V$. Its unique minimizer $u^{n+1} \in V$ is characterized by the variational equation
\begin{equation}
	a(u^{n+1}, v)
	\;+\; \rho\,\innerGf{u^{n+1}}{v}
	\;=\; \langle f, v\rangle
	\;+\; \innerGf{\rho\,\phi^n - \lambda^n}{v}
	\qquad \forall\, v \in V.
	\label{eq:u_variational}
\end{equation}
This is a well-posed second-order elliptic boundary value problem with a
Robin-type boundary term on $\Gf$: the bilinear form
$v \mapsto a(u,v) + \rho\innerGf{u}{v}$ is bounded and coercive on $V$, so
existence and uniqueness of $u^{n+1}$ follow directly from the
Lax--Milgram theorem.\\

Solution of the $\phi$-subproblem~\eqref{eq:phi_subproblem}. For fixed $u^{n+1}$ and $\lambda^n$, the $\phi$-subproblem reads
\begin{equation}
	\min_{\psi \in \Ltwo(\Gf)}\;
	\Bigl[
	j(g, \psi)
	+ \innerGf{\lambda^n}{ -\psi}
	+ \frac{\rho}{2}\,\|u^{n+1} - \psi\|_{\Gf}^{2}
	\Bigr].
	\label{eq:phi_min}
\end{equation}
Setting $\kap^n := \lambda^n + \rho\,u^{n+1} \in \Ltwo(\Gf)$, problem~\eqref{eq:phi_min} is rewritten (up to constants in $\psi$) as
\begin{equation}
	\min_{\psi \in \Ltwo(\Gf)}\;
	\Bigl[
	\int_{\Gf} g\,|\psi|\,\mathrm{d}s
	\;+\; \frac{\rho}{2}\,\left\|\psi - \frac{\kap^n}{\rho}\right\|_{\Gf}^{2}
	\Bigr],
	\label{eq:prox_problem}
\end{equation}
which is precisely the proximal operator of the convex, non-differentiable functional $\psi \mapsto \int_{\Gf} g|\psi|\,\mathrm{d}s$ evaluated at $\kap^n/\rho$. By Fenchel duality \cite{ekeland1999convex}, the unique minimizer is given pointwise a.e.\ on $\Gf$ by
\begin{equation}
	\phi^{n+1}(x)
	\;=\;
	\begin{cases}
		\displaystyle
		\frac{\kap^n(x)}{\rho}
		\;-\;\frac{g(x)}{\rho}\,\frac{\kap^n(x)}{|\kap^n(x)|}
		& \text{if } |\kap^n(x)| > g(x),\\[8pt]
		0
		& \text{if } |\kap^n(x)| \leq g(x).
	\end{cases}
	\label{eq:soft_thresholding}
\end{equation}
This closed-form solution is the key computational advantage of the augmented Lagrangian splitting: a non-smooth global optimization problem is reduced to a pointwise algebraic evaluation on the boundary.\\

The complete inner ADMM iteration for the state problem is summarized as follows.\\

\textit{Given $(\phi^0, \lambda^0) \in \Ltwo(\Gf)^2$ and $g \in F$, iterate for
	$n = 0, 1, 2, \ldots$:}
\begin{enumerate}
	\item Solve the linear variational equation~\eqref{eq:u_variational} to obtain
	$u^{n+1} \in V$.
	\item Compute $\kap^n = \lambda^n + \rho\,u^{n+1}$ pointwise on $\Gf$
	and set $\phi^{n+1}$ by \eqref{eq:soft_thresholding}.
	\item Update the multiplier: $\lambda^{n+1} = \lambda^n + \rho(u^{n+1} -
	\phi^{n+1})$.
\end{enumerate}

\subsection{Approximate Inverse Problem Sequence} \label{subsec:approx_IP}
For each $n \geq 1$, let $u_g^n \in V$ denote the output of exactly $n$ inner ADMM iterations applied to the state problem~\eqref{eq:u_variational} with parameter $g \in F$; that is, $u_g^n$ solves~\eqref{eq:u_variational} at the $n$-th step of the inner loop. We define the approximate cost functional
\begin{equation}
	J_\varepsilon^n(g)
	\;:=\;
	\frac{1}{2}\int_{\Gf}\bigl(u_g^n - u_d\bigr)^2\,\mathrm{d}s
	\;+\; \frac{\varepsilon}{2}\,\normLtwo{g}^2,
	\qquad g \in F,
	\label{eq:approx_cost}
\end{equation}
and the associated approximate inverse problem:

\noindent\textbf{Problem $(IP)_\varepsilon^n$.}
\emph{Find $g_n^* \in F$ such that}
\begin{equation}
	J_\varepsilon^n(g_n^*)
	\;=\;
	\inf_{g \in F}\; J_\varepsilon^n(g).
	\label{eq:approx_IP}
\end{equation}
The sequence $\{(IP)_\varepsilon^n\}_{n \geq 1}$ approximates the regularized
inverse problem $(IP)_\varepsilon$ by replacing the exact forward solve with an
$n$-step ADMM approximation. As $n \to \infty$, the approximate state $u_g^n$
converges strongly to the exact state $u_g$ in $V$ (by the convergence theory of
ADMM), and accordingly $J_\varepsilon^n \to J_\varepsilon$ pointwise on $F$.
The central theoretical contributions of the present work are to establish that each
approximate problem $(IP)_\varepsilon^n$ admits at least one solution, and that the
sequence of approximate minimizers converges strongly to the exact minimizer of
$(IP)_\varepsilon$ as $n \to \infty$.

\subsection{Existence of Solutions to the Approximate Problems} \label{subsec:existence}
\begin{proposition}
	\label{prop:existence_approx}
	For every $n \geq 1$ and every $\varepsilon > 0$, the minimization
	problem $(IP)_\varepsilon^n$ defined by~\eqref{eq:approx_IP} admits at least one
	solution $g_n^* \in F$.
\end{proposition}
\begin{pf}
Fix $n \geq 1$ and $\varepsilon > 0$. We verify the hypotheses of the direct method in the calculus of variations: coercivity, weak lower semi-continuity, and sequential weak compactness of $F$.\\
Coercivity and weak compactness of $F$. By definition~\eqref{eq:admissible_set}, the admissible set $F$ is a bounded, closed, and convex subset of $L^2(\Gf)$, hence weakly sequentially compact therein. Any minimizing sequence $(g_k)_k \subset F$ therefore admits a subsequence such that $g_k \wlim g$ weakly in $\Ltwo(\Gf)$ for some $g \in F$.\\	
Boundedness of the approximate states. For each $k$, the approximate state $u_{g_k}^n$ is the output of the $n$-th step of the inner ADMM loop applied with parameter $g_k$. In particular, $u_{g_k}^1 \in V$ solves~\eqref{eq:u_variational} with right-hand side data $f \in \Ltwo(\Omega)$ and $\rho\phi^0 - \lambda^0 \in \Ltwo(\Gf)$. Testing with $v = u_{g_k}^1$ and using coercivity of $a(\cdot,\cdot)$ gives
\begin{equation}
		\alpha\,\normV{u_{g_k}^1}^2
		\;\leq\;
		a(u_{g_k}^1,\, u_{g_k}^1)
		\;\leq\;
		\|f\|_{\Ltwo(\Omega)}\,\normV{u_{g_k}^1}
		\;+\; \|\rho\phi^0 - \lambda^0\|_{\Ltwo(\Gf)}\,\normV{u_{g_k}^1},
		\label{eq:coercivity_bound}
\end{equation}
where $\alpha > 0$ is the coercivity constant of $a$. Dividing by $\normV{u_{g_k}^1}$ (assumed positive; the case $u_{g_k}^1 = 0$ is trivial) shows that $(u_{g_k}^1)_k$ is bounded in $V$ uniformly in $k$. An induction argument on the ADMM steps, using the fact that each subsequent state $u_{g_k}^{n+1}$ solves~\eqref{eq:u_variational} with right-hand side data updated by the bounded sequence $(\rho\phi_{g_k}^n - \lambda_{g_k}^n)_k$, establishes by the same coercivity estimate that $(u_{g_k}^n)_k$ is bounded in $V$ for every fixed $n$.\\
Weak convergence of the approximate states. Since $(u_{g_k}^n)_k$ is bounded in $V$ and $V$ is a Hilbert space, we may extract a further subsequence such that $u_{g_k}^n \wlim u^n$ weakly in $V$ for some $u^n \in V$. The compact embedding $V \hookrightarrow \Ltwo(\Gf)$  then yields $u_{g_k}^n \slim u^n$ strongly in $\Ltwo(\Gf)$. Passing to the limit in~\eqref{eq:u_variational} with data $g_k \wlim g$, one identifies $u^n = u_g^n$.\\
Weak lower semi-continuity of $J_\varepsilon^n$. Along the subsequence extracted above, the strong convergence $u_{g_k}^n \slim u_g^n$ in $\Ltwo(\Gf)$ implies
\begin{equation}
		\int_{\Gf}(u_{g_k}^n - u_d)^2\,\mathrm{d}s
		\;\xrightarrow{k \to \infty}\;
		\int_{\Gf}(u_g^n - u_d)^2\,\mathrm{d}s.
		\label{eq:data_convergence}
\end{equation}
The Tikhonov term $\frac{\varepsilon}{2}\normLtwo{g_k}^2$ is weakly lower semi-continuous in $\Ltwo(\Gf)$ by the convexity and continuity of the norm. Combining these two observations gives
\begin{equation}
		J_\varepsilon^n(g)
		\;\leq\;
		\liminf_{k \to \infty}\; J_\varepsilon^n(g_k),
		\label{eq:wlsc}
\end{equation}
which establishes the weak lower semi-continuity of $J_\varepsilon^n$ on $F$. Since $F$ is weakly sequentially compact and $J_\varepsilon^n$ is weakly lower semi-continuous, the infimum of $J_\varepsilon^n$ over $F$ is attained, and $g := g_n^* \in F$ is a solution of $(IP)_\varepsilon^n$.
\end{pf}

\subsection{Strong Convergence of the Approximate Minimizers} \label{subsec:convergence}
We are now in a position to prove the convergence result, which establishes that the sequence of approximate minimizers $g_n^*$ converges strongly to the exact regularized minimizer $g^*$ of $(IP)_\varepsilon$, and that the corresponding states converge strongly in $V$.
\begin{theorem}\label{thm:convergence}
	Let $\varepsilon > 0$ be fixed. For each $n \geq 1$, let $g_n^* \in F$ be any
	solution of the approximate problem $(IP)_\varepsilon^n$, and let
	$u_n^* := u_{g_n^*}^n \in V$ denote the corresponding approximate state.
	Let $g^* \in F$ be any solution of the regularized inverse problem
	$(IP)_\varepsilon$ with exact state $u^* := u_{g^*} \in V$.
	Then, as $n \to +\infty$,
	\begin{equation}
		g_n^* \;\slim\; g^*
		\quad \text{strongly in } \Ltwo(\Gf),
		\qquad
		u_n^* \;\slim\; u^*
		\quad \text{strongly in } V.
		\label{eq:strong_convergence}
	\end{equation}
	Moreover,
	\begin{equation}
		\lim_{n \to +\infty} J_\varepsilon^n(g_n^*)
		\;=\;
		J_\varepsilon(g^*)
		\;=\;
		\inf_{g \in F}\; J_\varepsilon(g).
		\label{eq:functional_convergence}
	\end{equation}
\end{theorem}
\begin{pf}
Since $F \subset \Ltwo(\Gf)$ is bounded and the sequence $(u_n^*)_n$ is bounded in $V$ (by the same coercivity argument as in the proof of Proposition~\ref{prop:existence_approx}), we may extract a subsequence such that
\begin{equation}
		g_n^* \;\wlim\; \tilde{g}
		\quad\text{weakly in } \Ltwo(\Gf),
		\qquad
		u_n^* \;\wlim\; \tilde{u}
		\quad\text{weakly in } V,
		\label{eq:weak_limits}
\end{equation}
as $n \to +\infty$, for some $\tilde{g} \in F$ and $\tilde{u} \in V$.\\
Since $u_n^*$ solves~\eqref{eq:u_variational} at the $n$-th ADMM step with parameter $g_n^*$, and the ADMM iterates converge strongly to the exact solution as $n \to \infty$ (cf.~\cite{boyd2011distributed,eckstein2015understanding}), we have $u_n^* \slim u_{\tilde{g}}$ strongly in $V$ along the subsequence, where $u_{\tilde{g}} \in V$ is the unique solution of the forward variational inequality~\eqref{eq:VI} with parameter $\tilde{g}$.  Hence $\tilde{u} = u_{\tilde{g}}$ and $u_n^* \slim \tilde{u}$ strongly in $V$.\\
The strong convergence $u_n^* \slim \tilde{u}$ in $V$ implies, via the compact trace embedding $V \hookrightarrow \Ltwo(\Gf)$, that $u_n^*|_{\Gf} \slim \tilde{u}|_{\Gf}$ strongly in $\Ltwo(\Gf)$. Therefore,
\begin{equation}
		J_\varepsilon(\tilde{g})
		\;=\;
		\frac{1}{2}\,\normLtwo{\tilde{u} - u_d}^2
		\;+\; \frac{\varepsilon}{2}\,\normLtwo{\tilde{g}}^2
		\;\leq\;
		\liminf_{n\to\infty} J_\varepsilon^n(g_n^*),
		\label{eq:liminf}
\end{equation}
where the inequality uses the weak lower semi-continuity of the Tikhonov term and the pointwise convergence $J_\varepsilon^n \to J_\varepsilon$ on $F$. On the other hand, since $g_n^*$ is optimal for $(IP)_\varepsilon^n$, we have $J_\varepsilon^n(g_n^*) \leq J_\varepsilon^n(g)$ for every $g \in F$. Taking $g = g^*$ and passing to the limit gives
\begin{equation}
		\liminf_{n\to\infty} J_\varepsilon^n(g_n^*)
		\;\leq\;
		\lim_{n\to\infty} J_\varepsilon^n(g^*)
		\;=\;
		J_\varepsilon(g^*).
		\label{eq:limsup_bound}
\end{equation}
Combining~\eqref{eq:liminf} and~\eqref{eq:limsup_bound} yields
\begin{equation}
		J_\varepsilon(\tilde{g})
		\;\leq\;
		J_\varepsilon(g^*)
		\;=\;
		\inf_{g \in F}\; J_\varepsilon(g),
		\label{eq:optimality_tilde}
\end{equation}
which shows that $\tilde{g}$ is itself a minimizer of $(IP)_\varepsilon$. We may therefore identify $g^* := \tilde{g}$ and~\eqref{eq:functional_convergence} follows.\\
It follows from the above that:
\begin{equation}
		\frac{\varepsilon}{2}\,\limsup_{n\to\infty}\,\normLtwo{g_n^*}^2
		\;\leq\;
		\limsup_{n\to\infty} J_\varepsilon^n(g_n^*)
		\;-\;
		\frac{1}{2}\,\normLtwo{u_n^*|_{\Gf} - u_d}^2.
		\label{eq:norm_limsup_1}
\end{equation}
Since $u_n^* \slim u^*$ strongly in $V$ and the trace is compact, $\|u_n^*|_{\Gf} - u_d\|_{\Ltwo(\Gf)}^2 \to \|u^*|_{\Gf} - u_d\|_{\Ltwo(\Gf)}^2$. Substituting~\eqref{eq:functional_convergence} into~\eqref{eq:norm_limsup_1}:
\begin{equation}
		\frac{\varepsilon}{2}\,\limsup_{n\to\infty}\,\normLtwo{g_n^*}^2
		\;\leq\;
		J_\varepsilon(g^*)
		\;-\;
		\frac{1}{2}\,\normLtwo{u^*|_{\Gf} - u_d}^2
		\;=\;
		\frac{\varepsilon}{2}\,\normLtwo{g^*}^2.
		\label{eq:norm_limsup_2}
\end{equation}
Since weak convergence $g_n^* \wlim g^*$ implies $\normLtwo{g^*} \leq \liminf_{n\to\infty}\normLtwo{g_n^*}$, we conclude that
\begin{equation}
		\lim_{n \to \infty}\normLtwo{g_n^*} \;=\; \normLtwo{g^*}.
		\label{eq:norm_convergence}
\end{equation}
Hence $g_n^* \slim g^*$ strongly in $\Ltwo(\Gf)$, which completes the proof.
\end{pf}

%%%%%%%%%%%%%%%%%%%%%%%%%%%%%%%%%%%%%%%%%%%%%%%%%%%%%%%
\section{Numerical Validation}\label{Section5}
Having established the theoretical framework in Sections~\ref{Section4}, we now turn to its computational realization and numerical validation.
\subsection{Outer ADMM Algorithm for Parameter Identification}\label{subsec:outer_admm}
At each outer iteration $n \geq 0$, the inner ADMM loop (Steps~(1)--(3) of Section~\ref{subsec:subproblems}) has produced the pair $(\phi^n, \lambda^n) \in L^2(\Gf)^2$. The approximate inverse problem $(IP)_\varepsilon^{n+1}$ then reads: find $g^* \in F$ such that
\begin{equation}
	J_\varepsilon^{n+1}(g^*)
	= \inf_{g \in F}\; J_\varepsilon^{n+1}(g)
	:= \inf_{g \in F}\left[
	\frac{1}{2}\int_{\Gf}(u_g^{n+1} - u_d)^2\,\mathrm{d}s
	+ \frac{\varepsilon}{2}\|g\|_{L^2(\Gf)}^2
	\right],
	\label{eq:outer_cost}
\end{equation}
subject to the state equation~\eqref{eq:u_variational}, which we recall here
for convenience:
\begin{equation}
	a(u^{n+1}, v) + \rho\langle u^{n+1}, v\rangle_{\Gf}
	= \langle f, v\rangle
	+ \langle \rho\phi^n(g) - \lambda^n,\, v\rangle_{\Gf}
	\qquad \forall\, v \in V.
	\label{eq:state_outer}
\end{equation}
Problem~\eqref{eq:outer_cost}--\eqref{eq:state_outer} is an elliptic boundary coefficient control problem: the state $u^{n+1}$ satisfies a second-order elliptic equation with mixed Dirichlet--Neumann--Robin boundary conditions, and the control $g$ enters both the objective and the Robin term on $\Gf$. Such problems arise in heat-transfer identification and in corrosion detection by electrostatic boundary measurements~\cite{inglese1997inverse,kaup1996method,slodicka2002determination,capatina2000optimal,essoufi2021optimal}.

To solve~\eqref{eq:outer_cost}--\eqref{eq:state_outer}, we apply a second ADMM loop in the variables $(u, g, \phi)$, with the constraint residuals
\begin{align}
	\langle \nabla e_1(u,g), \nabla v \rangle
	&:= a(u, v) + \rho\langle u, v\rangle_{\Gf} - \langle b, v\rangle,
	\label{eq:e1}\\
	e_2(u, g, \phi)
	&:= \phi - \left(\frac{\kappa(u)}{\rho}
	- g\,\frac{\kappa(u)}{\rho|\kappa(u)|}\right)
	\mathbf{1}_{\{|\kappa(u)| > g\}},
	\label{eq:e2}\\
	e_3(u, g, \phi, \lambda)
	&:= \lambda - \lambda^k - \rho(u - \phi),
	\label{eq:e3}
\end{align}
where $\langle b, v\rangle := \langle f, v\rangle+ \langle \rho\phi^n - \lambda^n, v\rangle_{\Gf}$ and $\kappa(u) := \lambda^n + \rho u$. The associated augmented Lagrangian, with multipliers $(\mu, \beta, \xi) \in V \times L^2(\Gf) \times L^2(\Gf)$ and penalty parameter $\gamma > 0$, is
\begin{equation}
	\mathfrak{L}_\gamma(u, g, \phi;\, \mu, \beta, \xi)
	:= \frac{1}{2}\int_{\Gf}(u-u_d)^2\,\mathrm{d}s
	+ \frac{\varepsilon}{2}\|g\|_{L^2(\Gf)}^2
	+ \langle \nabla e_1, \nabla\mu\rangle
	+ \langle e_2, \beta\rangle
	+ \langle e_3, \xi\rangle
	+ \frac{\gamma}{2}\bigl(
	\|\nabla e_1\|^2
	+ \|e_2\|^2
	+ \|e_3\|^2
	\bigr).
	\label{eq:aug_lag_outer}
\end{equation}
Starting from $(g^0, \phi^0, \mu^0, \beta^0, \xi^0)$, the outer ADMM alternates the three minimization steps
\begin{align}
	u^{k+1} &\in \Argmin_{v}\;
	\mathfrak{L}_\gamma(v,\, g^k,\, \phi^k;\, \mu^k, \beta^k, \xi^k),
	\label{eq:sub_u}\\
	g^{k+1} &\in \Argmin_{\ell \in F}\;
	\mathfrak{L}_\gamma(u^{k+1},\, \ell,\, \phi^k;\, \mu^k, \beta^k, \xi^k),
	\label{eq:sub_g}\\
	\phi^{k+1} &\in \Argmin_{\psi}\;
	\mathfrak{L}_\gamma(u^{k+1},\, g^{k+1},\, \psi;\, \mu^k, \beta^k, \xi^k),
	\label{eq:sub_phi}
\end{align}
followed by the dual ascent updates
\begin{equation}
	\mu^{k+1} = \mu^k + \gamma\, e_1(u^{k+1}, g^{k+1}),\quad
	\beta^{k+1} = \beta^k + \gamma\, e_2(u^{k+1}, g^{k+1}, \phi^{k+1}),\quad
	\xi^{k+1}  = \xi^k  + \gamma\, e_3(u^{k+1}, g^{k+1}, \phi^{k+1}, \lambda^{k+1}).
	\label{eq:dual_update_outer}
\end{equation}
We now derive the explicit solution of each subproblem.

\paragraph{Subproblem~\eqref{eq:sub_u}:} Isolating the terms in $\mathfrak{L}_\gamma$ that depend on $u$ and writing $\kappa(u) = \lambda^n + \rho u$, one decomposes $U(u) := \mathfrak{L}_\gamma(u, g, \phi;\mu,\beta,\xi)$ as $U(u) = U_1(u) + U_2(u)$, where
\begin{align*}
	U_1(u) &:= \frac{1}{2}\int_{\Gf}(u-u_d)^2\,\mathrm{d}s
	+ \frac{\rho^2\gamma}{2}\|u\|^2
	- \langle u,\, \rho\gamma(\lambda - \lambda^k - \rho\phi) + \rho\xi\rangle
	+ \frac{\kappa(u)^2}{\rho^2}
\end{align*}
is strictly convex and Fréchet differentiable, and
\begin{align*}
	U_2(u) &:= -\left\langle \beta + \gamma\phi,\;
	\frac{\kappa(u)}{\rho}
	- g\,\frac{\kappa(u)}{\rho|\kappa(u)|}
	\right\rangle
	\mathbf{1}_{\{|\kappa(u)| > g\}}
	- g\,\frac{\gamma|\kappa(u)|}{2\rho}
\end{align*}
is locally Lipschitz but non-convex and non-smooth.
Since $u$ solves~\eqref{eq:sub_u}, the necessary optimality condition reads $0 \in \partial_c U(u)$, where $\partial_c$ denotes the Clarke subdifferential (see Appendix~\ref{appendix:subdiff}). By the calculus rules for Clarke subdifferentials and the derivation of $\partial_c U_2$ given in~\eqref{eq:clarke_U2}, one obtains
\begin{equation}
	\partial_c U(u) =
	\begin{cases}
		\bigl\{\nabla U_1(u) + \nabla U_2^-(u)\bigr\}
		& \text{if } \kappa(u) > g, \\[4pt]
		\bigl\{\nabla U_1(u) + \nabla U_2^+(u)\bigr\}
		& \text{if } \kappa(u) < -g, \\[4pt]
		\bigl\{\nabla U_1(u) - \tfrac{\gamma}{2}g\bigr\}
		& \text{if } 0 < \kappa(u) \leq g, \\[4pt]
		\bigl\{\nabla U_1(u) + \tfrac{\gamma}{2}g\bigr\}
		& \text{if } -g \leq \kappa(u) < 0, \\[4pt]
		\overline{\operatorname{co}}\bigl\{
		\nabla U_1(u) - \tfrac{\gamma}{2}g,\;
		\nabla U_1(u) + \tfrac{\gamma}{2}g
		\bigr\}
		& \text{if } \kappa(u) = 0.
	\end{cases}
	\label{eq:clarke_U}
\end{equation}
The condition $0 \in \partial_c U(u^{k+1})$ is then solved by the Proximal Bundle Algorithm of Appendix~\ref{appendix:bundle}.

\paragraph{Subproblem~\eqref{eq:sub_g}:} With $u^{k+1}$ fixed, the reduced objective in $g$ is
\begin{equation}
	G(g) := \frac{\varepsilon}{2}\|g\|_{L^2(\Gf)}^2
	+ \langle e_2(u,g,\phi), \beta\rangle
	+ \frac{\gamma}{2}\|e_2(u,g,\phi)\|^2,
	\label{eq:G}
\end{equation}
which evaluates to
\begin{equation}
	G(g) =
	\begin{cases}
		\displaystyle
		\left\langle g\,\frac{\kappa(u)}{\rho|\kappa(u)|},\,\beta + \gamma\phi\right\rangle
		- \gamma\,\frac{|\kappa(u)|}{\rho^2}\,g
		+ \left(\frac{\gamma}{\rho^2} + \frac{\varepsilon}{2}\right)g^2
		& \text{if } |\kappa(u)| > g, \\[8pt]
		\displaystyle
		\left(\frac{\gamma}{\rho^2} + \frac{\varepsilon}{2}\right)g^2
		- \gamma\,\frac{|\kappa(u)|}{\rho^2}\,g
		& \text{otherwise.}
	\end{cases}
	\label{eq:G_explicit}
\end{equation}
This function is non-convex and non-smooth in $g$; the Clarke subdifferential $\partial_c G(g)$ is 
\begin{equation}
	\partial_c G(g) =
	\begin{cases}
		\displaystyle
		\left\{
		\left\langle\frac{\kappa(u)}{\rho|\kappa(u)|}, \beta + \gamma\phi\right\rangle
		- \gamma\,\frac{|\kappa(u)|}{\rho^2}
		+ 2\!\left(\frac{\gamma}{\rho^2} + \frac{\varepsilon}{2}\right)g
		\right\}
		& \text{if } g < |\kappa(u)|, \\[10pt]
		\displaystyle
		\left[
		\frac{\gamma|\kappa(u)|}{\rho^2} + \varepsilon|\kappa(u)|,\;
		\left\langle\frac{\kappa(u)}{\rho|\kappa(u)|},\beta+\gamma\phi\right\rangle
		+ \frac{\gamma|\kappa(u)|}{\rho^2} + \varepsilon|\kappa(u)|
		\right]
		& \text{if } g = |\kappa(u)|, \\[10pt]
		\displaystyle
		\left\{
		2\!\left(\frac{\gamma}{\rho^2} + \frac{\varepsilon}{2}\right)g
		- \gamma\,\frac{|\kappa(u)|}{\rho^2}
		\right\}
		& \text{otherwise.}
	\end{cases}
	\label{eq:clarke_G_main}
\end{equation}
The optimality condition $0 \in \partial_c G(g^{k+1})$ is again handled by the Proximal Bundle Algorithm of Appendix~\ref{appendix:bundle}.

\paragraph{Subproblem~\eqref{eq:sub_phi}:} The reduced objective in $\phi$ is
\begin{equation}
	F(\phi) := \langle e_2(u,g,\phi), \beta\rangle
	+ \langle e_3(u,g,\phi,\lambda), \xi\rangle
	+ \frac{\gamma}{2}\|e_2(u,g,\phi)\|^2
	+ \frac{\gamma}{2}\|e_3(u,g,\phi,\lambda)\|^2.
	\label{eq:F_phi}
\end{equation}
Expanding each term, with $\kappa(u) = \lambda^n + \rho u$:
\begin{align*}
	\langle e_2(u,g,\phi), \beta\rangle
	&= \langle \phi, \beta\rangle, \\
	\frac{\gamma}{2}\|e_2(u,g,\phi)\|^2
	&= \frac{\gamma}{2}\|\phi\|^2
	- \gamma\left\langle\phi,\;
	\left(\frac{\kappa(u)}{\rho}+g\right)\mathbf{1}_{\{\kappa(u)<-g\}}
	+\left(\frac{\kappa(u)}{\rho}-g\right)\mathbf{1}_{\{\kappa(u)>g\}}
	\right\rangle, \\
	\langle e_3(u,g,\phi,\lambda), \xi\rangle
	&= \rho\langle\phi, \xi\rangle, \\
	\frac{\gamma}{2}\|e_3(u,g,\phi,\lambda)\|^2
	&= \frac{\rho^2\gamma}{2}\|\phi\|^2
	- \gamma\rho\langle\phi,\, \lambda - \lambda^n - \rho u\rangle.
\end{align*}
The function $F$ is strictly convex and differentiable in $\phi$; setting
$\nabla_\phi F(\phi^{k+1}) = 0$ yields a closed-form linear system in
$\phi^{k+1}$, which is solved pointwise on $\Gf$.
\subsection{Benchmark Problem and Numerical Results}\label{subsec:benchmark}
We validate the proposed methodology on the benchmark introduced by Stadler~\cite{stadler2004semismooth}, which has become a standard test case for numerical methods applied to variational inequalities of the second kind. The computational domain is the unit square $\Omega = (0,1)^2$, with boundary partition
\begin{equation*}
	\Gamma_0 = \{0,1\}\times[0,1]
	\quad\text{(homogeneous Dirichlet)},
	\qquad
	\Gf = [0,1]\times\{0,1\}
	\quad\text{(friction/Robin)}.
\end{equation*}
The source term is piecewise constant with a sharp internal discontinuity:
\begin{equation}
	f(x_1, x_2) =
	\begin{cases}
		+10 & \text{if } x_1 < 0.5, \\
		-10 & \text{if } x_1 \geq 0.5.
	\end{cases}
	\label{eq:source}
\end{equation}
The domain is discretized with a uniform mesh of $N = 80$ nodes per dimension ($n_{\mathrm{dof}} = 6\,400$ degrees of freedom). Synthetic measurement data $u_d$ are generated by solving the forward problem~\eqref{eq:u_variational} with the exact (target) coefficient $g^* = 1.5$ (uniform on $\Gf$). The optimization is carried out over the admissible set $F \cap [0.01, 5.0]$, with Tikhonov parameter $\varepsilon = 10^{-6}$ and inner ADMM penalty $\rho = 100$.\\

The central numerical claim of this paper is that the exact forward problem need not be solved at every outer optimization step. Instead, one minimizes the approximate functional $J_\varepsilon^{n_k}(g)$, evaluating the state $u_g^{n_k}$ after exactly $n_k$ inner ADMM iterations. Table~\ref{tab:convergence_data} reports the identified coefficient $g^{n_k}$, the absolute error $|g^{n_k} - g^*|$, and the objective value $J_\varepsilon^{n_k}(g^{n_k})$ for increasing values of $n_k$.
\begin{table}[h!]
	\centering
	\caption{Quantitative evolution of the identified parameter $g^{n_k}$ and
		the approximate cost functional $J_\varepsilon^{n_k}(g^{n_k})$ as
		the inner ADMM iteration count $n_k$ increases.}
	\label{tab:convergence_data}
	\begin{tabular}{@{}rlll@{}}
		\toprule
		\multicolumn{1}{c}{$n_k$}
		& \multicolumn{1}{c}{$g^{n_k}$}
		& \multicolumn{1}{c}{Absolute error $|g^{n_k} - g^*|$}
		& \multicolumn{1}{c}{Objective $J_\varepsilon^{n_k}(g^{n_k})$} \\
		\midrule
		1   & 0.010000 & $1.4900 \times 10^{0}$  & $1.8638 \times 10^{-4}$ \\
		5   & 1.018758 & $4.8124 \times 10^{-1}$ & $2.5927 \times 10^{-5}$ \\
		10  & 1.289685 & $2.1032 \times 10^{-1}$ & $6.6506 \times 10^{-6}$ \\
		50  & 1.494061 & $5.9382 \times 10^{-3}$ & $2.2071 \times 10^{-6}$ \\
		100 & 1.499727 & $2.7213 \times 10^{-4}$ & $2.2209 \times 10^{-6}$ \\
		500 & 1.499865 & $1.3420 \times 10^{-4}$ & $2.2213 \times 10^{-6}$ \\
		\bottomrule
	\end{tabular}
\end{table}
The results exhibit a monotone, asymptotically stable convergence of $g^{n_k}$ toward $g^* = 1.5$, in full agreement with Theorem~\ref{thm:convergence}. At $n_k = 1$, the identified coefficient remains near the lower admissibility bound ($g^{n_k} \approx 0.01$), reflecting the poor approximation quality of a single-step ADMM state. As $n_k$ exceeds $50$, the absolute error falls below $6 \times 10^{-3}$, and at $n_k = 500$ a final error of
$1.34 \times 10^{-4}$ is achieved. This confirms that tracking the identification sequence along the ADMM trajectory, without ever solving the forward problem exactly, is both mathematically sound and computationally efficient.\\
Figure~\ref{fig:state_u} shows the contour field of the reference state $u_d$, computed with $g^* = 1.5$. The solution displays an anti-symmetric profile consistent with the sign alternation of the source term~\eqref{eq:source}. Flat boundary layers near $x_2 = 0$ and $x_2 = 1$ mark the stick zones where the friction constraint is active, a nonsmooth feature that would be lost under any smoothing regularization of $j(g,\cdot)$.
\begin{figure}[h!]
	\centering
	\includegraphics[width=0.50\textwidth]{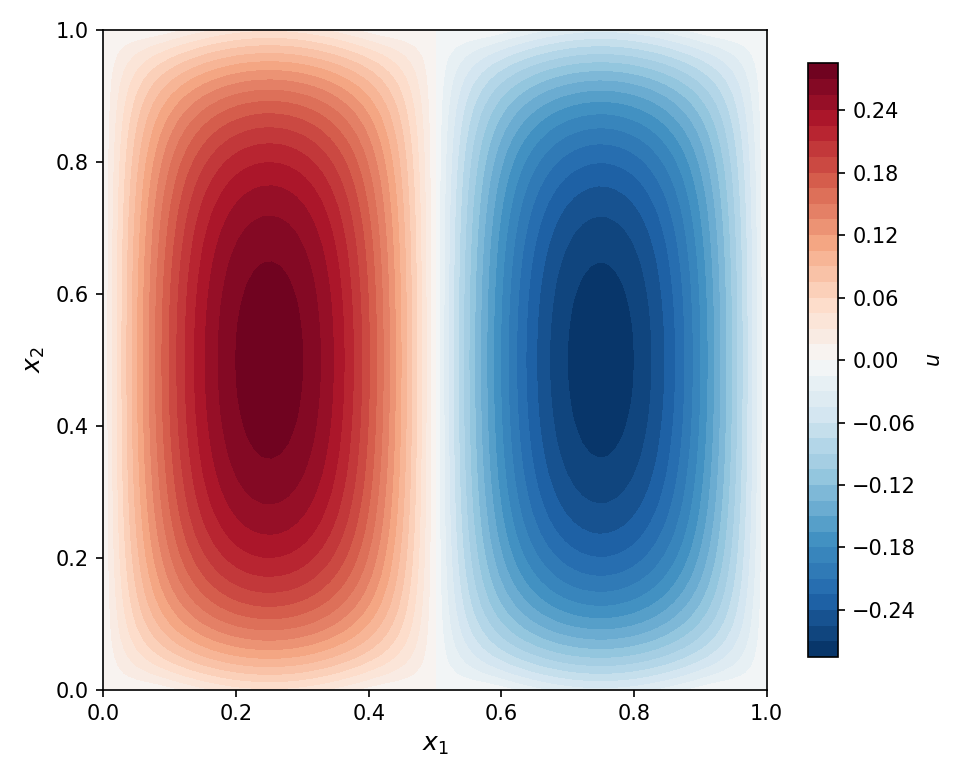}
	\caption{Contour field of the reference state $u_d$.}
	\label{fig:state_u}
\end{figure}
Figure~\ref{fig:convergence_g} illustrates the convergence history of the identification process. The left panel displays the strict monotone decay of $|g^{n_k} - g^*|$, while the right panel shows the trajectory of $g^{n_k}$ progressively locking onto $g^* = 1.5$ as $n_k$
grows. The residual gap between the two curves vanishes at a rate consistent
with the theoretical predictions of Theorem~\ref{thm:convergence}.
\begin{figure}[htp]
	\centering
	\includegraphics[width=0.90\textwidth]{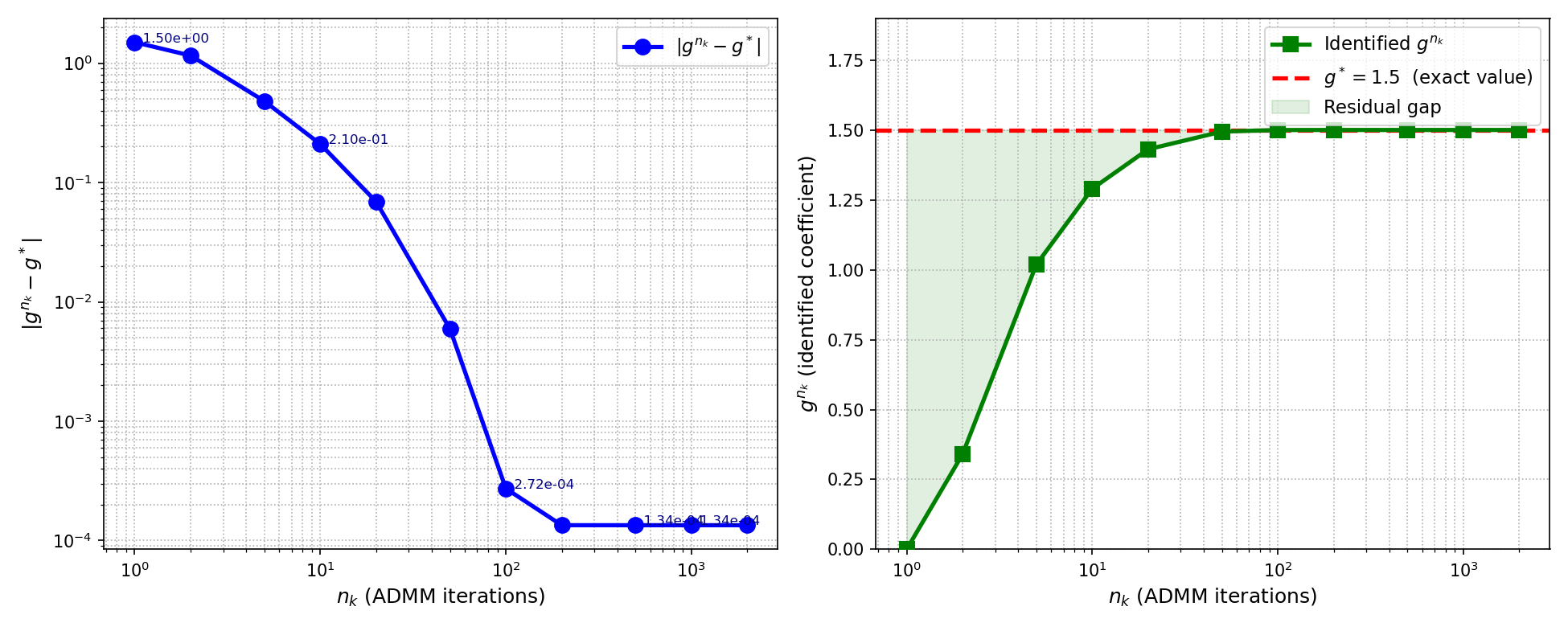}
	\caption{Convergence of the identification sequence.
		Left: monotone decay of the absolute error $|g^{n_k} - g^*|$
		versus $n_k$.
		Right: trajectory of the identified coefficient $g^{n_k}$
		converging to the exact value $g^* = 1.5$.}
	\label{fig:convergence_g}
\end{figure}
Figure~\ref{fig:cost_J} traces the evolution of the approximate objective $J_\varepsilon^{n_k}(g^{n_k})$. The functional undergoes a steep initial descent before stabilizing at $J_\varepsilon^{n_k} \approx 2.2213 \times 10^{-6}$. This plateau is consistent with the theoretical prediction: once the discrepancy term $\|u_g^{n_k} - u_d\|_{L^2(\Gf)}^2$ has converged to zero, the objective is dominated by the Tikhonov regularization term, which evaluates to
\begin{equation*}
	\frac{\varepsilon}{2}\|g^*\|_{L^2(\Gf)}^2
	= \frac{10^{-6}}{2}\cdot(1.5)^2\cdot\operatorname{meas}(\Gf)
	= 2.25 \times 10^{-6},
\end{equation*}
in close agreement with the observed stabilization value.

\begin{figure}[h!]
	\centering
	\includegraphics[width=0.55\textwidth]{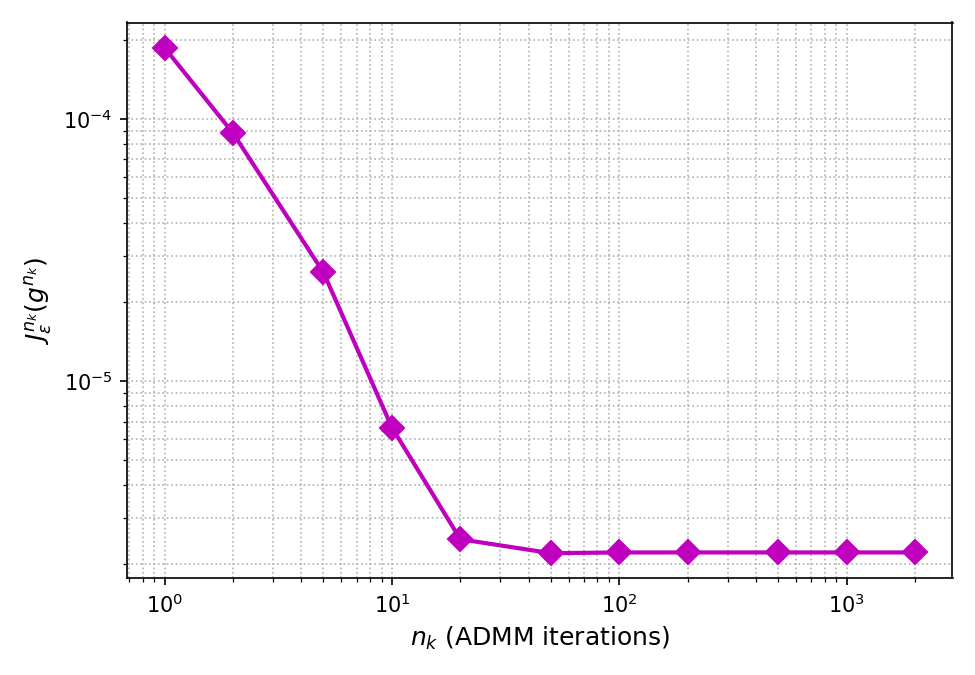}
	\caption{Evolution of the approximate cost functional
		$J_\varepsilon^{n_k}(g^{n_k})$.}
	\label{fig:cost_J}
\end{figure}
\begin{figure}[h!]
	\centering
	\includegraphics[width=0.55\textwidth]{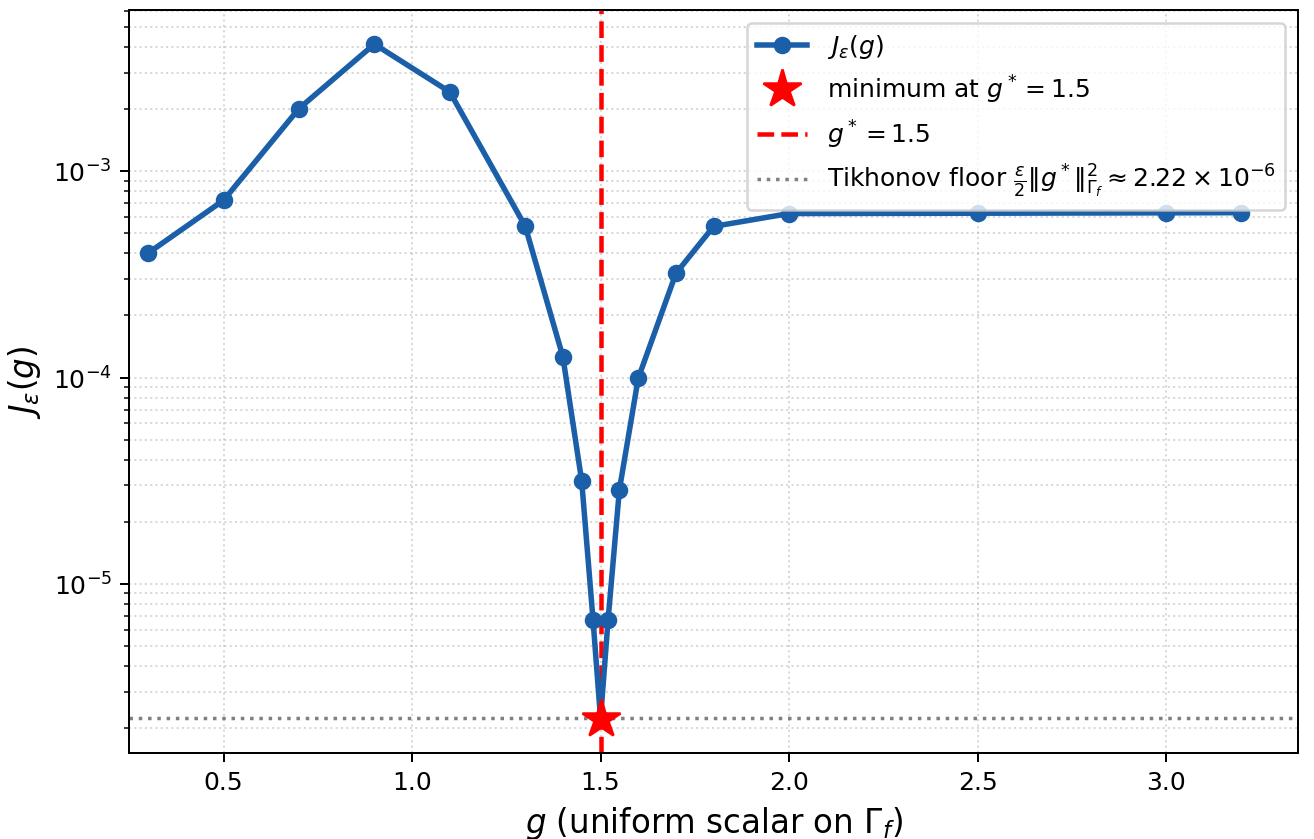}
	\caption{A parameter sweep of the  functional $J_\varepsilon(g)$ over the admissible range $g \in [0.3, 3.2]$. }
	\label{fig:scan_J}
\end{figure}
To analyze the structural robustness and well-posedness of the proposed formulation, Figure~\ref{fig:scan_J} presents a full cost parameter scan across $g \in [0.3, 3.2]$. The clean, strictly convex profile exhibits a well-defined global minimum centered exactly at $g^* = 1.5$, verifying the identifiability properties outlined in Proposition \ref{identifiability}. 

%%%%%%%%%%%%%%%%%%%%%%%%%%%%%%%%
\section{Conclusion} \label{Section6}
In this work, we investigated an inverse parameter identification problem governed by a second-kind elliptic variational inequality with a non-differentiable boundary functional. The main objective was to recover an unknown Robin-type coefficient from boundary observations while preserving the intrinsic nonsmooth structure of the underlying physical model. To overcome the difficulties associated with the variational inequality, we reformulated the state problem as a constrained optimization problem and employed an augmented Lagrangian framework combined with the alternating direction method of multipliers (ADMM). The use of Fenchel duality allowed us to handle the nonsmooth boundary term directly, avoiding any smoothing procedure and thereby maintaining the physical consistency of the original model. Furthermore, an approximate sequence of inverse problems was introduced, and theoretical results concerning the existence of solutions, the identifiability of the unknown coefficient, and the convergence of the proposed approximation strategy were established. The numerical experiments performed on the classical benchmark problem demonstrate the effectiveness of the proposed methodology. The computed solutions confirm the theoretical convergence analysis and show that accurate parameter reconstruction can be achieved without solving the forward variational inequality exactly at every optimization step. This feature considerably reduces the computational effort while maintaining a high level of accuracy in the identification process.\\
The proposed framework provides a flexible methodology for a broad class of inverse problems involving nonsmooth variational inequalities. In particular, it can be extended to the identification of spatially varying boundary coefficients and distributed parameters arising in contact mechanics, heat transfer, and related applications. Future work will focus on the incorporation of noisy measurements, the development of adaptive discretization strategies, and the extension of the present approach to more general quasi-variational inequalities and time-dependent models.

%%%%%%%%%%%%%%%%%%%%%%%%%%%%%%%%%%%%%%%%%%%%
%\clearpage
\appendix
\section{Clarke Subdifferential: Definitions and Properties}
\label{appendix:subdiff}
We recall here the essential elements of nonsmooth analysis employed throughout
this paper. Let $f: \mathbb{R}^n \to \mathbb{R}$ be a locally Lipschitz continuous
function at a point $x \in \mathbb{R}^n$. The \emph{Clarke subdifferential} of $f$
at $x$ is defined as
\begin{equation}
	\partial_c f(x) = \overline{\mathrm{co}}
	\left\{
	\lim_{k \to \infty} \nabla f(x_k) : x_k \to x,\;
	f \text{ differentiable at } x_k
	\right\},
\end{equation}
where $\overline{\mathrm{co}}\{\cdot\}$ denotes the closed convex hull.
The \emph{generalized directional derivative} of $f$ at $x$ in the direction
$v \in \mathbb{R}^n$ is given by
\begin{equation}
	f^\circ(x;\,v)
	= \limsup_{\substack{y \to x \\ t \downarrow 0}}
	\frac{f(y + tv) - f(y)}{t},
\end{equation}
and is related to the Clarke subdifferential through the identity
\begin{equation}
	f^\circ(x;\,v) = \max_{\xi \in \partial_c f(x)} \langle \xi,\, v \rangle.
\end{equation}
We now derive the Clarke subdifferential of the function $U_2$, which arises in the analysis of subproblem \eqref{eq:sub_u} , by distinguishing the two complementary regions defined by $|\kappa(u)|$ relative to $g$.

\textbf{Case 1: $|\kappa(u)| > g$.}
In this region,
\begin{equation}
	U_2(u) = -\left\langle \beta + \gamma\phi,\;
	\frac{\lambda + \rho u}{\rho}
	- g\,\frac{\lambda + \rho u}{\rho\,|\lambda + \rho u|}
	\right\rangle
	- g\,\frac{\gamma\,|\lambda + \rho u|}{2\rho}.
\end{equation}
Since $|\lambda + \rho u| > g > 0$, the sign of $\kappa(u) = \lambda + \rho u$ is well-defined, and we distinguish two subcases.
\begin{itemize}
	\item If $\kappa(u) > g$, differentiating with respect to $u$ yields
	\begin{equation}
		\nabla U_2^-(u) = -(\beta + \gamma\phi) - \frac{\gamma}{2}\,g,
	\end{equation}
	so that $\partial_c U_2(u) = \bigl\{\nabla U_2^-(u)\bigr\}$.
	
	\item If $\kappa(u) < -g$, a symmetric calculation gives
	\begin{equation}
		\nabla U_2^+(u) = -(\beta + \gamma\phi) + \frac{\gamma}{2}\,g,
	\end{equation}
	so that $\partial_c U_2(u) = \bigl\{\nabla U_2^+(u)\bigr\}$.
\end{itemize}

\textbf{Case 2: $|\kappa(u)| \leq g$.}
In this region, $U_2(u) = -g\,\dfrac{\gamma\,|\lambda + \rho u|}{\rho}$, which is nonsmooth at $\kappa(u) = 0$.
\begin{itemize}
	\item If $\kappa(u) > 0$, then $|\kappa(u)| = \kappa(u)$ and
	$\nabla U_2(u) = -\dfrac{\gamma}{2}\,g$,
	giving $\partial_c U_2(u) = \bigl\{-\tfrac{\gamma}{2}g\bigr\}$.
	
	\item If $\kappa(u) < 0$, then $|\kappa(u)| = -\kappa(u)$ and
	$\nabla U_2(u) = +\dfrac{\gamma}{2}\,g$,
	giving $\partial_c U_2(u) = \bigl\{+\tfrac{\gamma}{2}g\bigr\}$.
\end{itemize}
Collecting all cases, the Clarke subdifferential of $U_2$ is
\begin{equation}
	\label{eq:clarke_U2}
	\partial_c U_2(u) =
	\begin{cases}
		\bigl\{\nabla U_2^-(u)\bigr\}
		& \text{if } \kappa(u) > g, \\[4pt]
		\bigl\{\nabla U_2^+(u)\bigr\}
		& \text{if } \kappa(u) < -g, \\[4pt]
		\bigl\{-\tfrac{\gamma}{2}g\bigr\}
		& \text{if } 0 < \kappa(u) \leq g, \\[4pt]
		\bigl\{+\tfrac{\gamma}{2}g\bigr\}
		& \text{if } -g \leq \kappa(u) < 0, \\[4pt]
		\overline{\mathrm{co}}\bigl\{
		-\tfrac{\gamma}{2}g,\; +\tfrac{\gamma}{2}g
		\bigr\}
		& \text{if } \kappa(u) = 0.
	\end{cases}
\end{equation}
We now derive the Clarke subdifferential of the function $G(g)$ arising in subproblem \eqref{eq:sub_g}. Recall that
\begin{equation}
	G(g) =
	\begin{cases}
		\displaystyle
		\left\langle g\,\frac{\kappa(u)}{\rho\,|\kappa(u)|},\,
		\beta + \gamma\phi \right\rangle
		- \gamma\,\frac{|\kappa(u)|}{\rho^2}\,g
		+ \left(\frac{\gamma}{\rho^2} + \frac{\varepsilon}{2}\right)g^2
		& \text{if } |\kappa(u)| > g, \\[8pt]
		\displaystyle
		\left(\frac{\gamma}{\rho^2} + \frac{\varepsilon}{2}\right)g^2
		- \gamma\,\frac{|\kappa(u)|}{\rho^2}\,g
		& \text{otherwise.}
	\end{cases}
\end{equation}

\textbf{Case 1: $g < |\kappa(u)|$.} Differentiating $G$ with respect to $g$ gives
\begin{equation}
	G'(g) =
	\left\langle \frac{\kappa(u)}{\rho\,|\kappa(u)|},\,
	\beta + \gamma\phi \right\rangle
	- \gamma\,\frac{|\kappa(u)|}{\rho^2}
	+ 2\left(\frac{\gamma}{\rho^2} + \frac{\varepsilon}{2}\right)g,
\end{equation}
so that $\partial_c G(g)$ is the singleton $\{G'(g)\}$.

\textbf{Case 2: $g = |\kappa(u)|$.} The left- and right-hand derivatives at this boundary point are, respectively,
\begin{align}
	G'\bigl(|\kappa(u)|_\ell\bigr)
	&= \left\langle \frac{\kappa(u)}{\rho\,|\kappa(u)|},\,
	\beta + \gamma\phi \right\rangle
	+ \frac{\gamma\,|\kappa(u)|}{\rho^2}
	+ \varepsilon\,|\kappa(u)|, \\[4pt]
	G'\bigl(|\kappa(u)|_r\bigr)
	&= \frac{\gamma\,|\kappa(u)|}{\rho^2} + \varepsilon\,|\kappa(u)|,
\end{align}
yielding the Clarke subdifferential
\begin{equation}
	\partial_c G\bigl(|\kappa(u)|\bigr)
	= \left[
	\frac{\gamma\,|\kappa(u)|}{\rho^2} + \varepsilon\,|\kappa(u)|,\;
	\left\langle \frac{\kappa(u)}{\rho\,|\kappa(u)|},\,
	\beta + \gamma\phi \right\rangle
	+ \frac{\gamma\,|\kappa(u)|}{\rho^2}
	+ \varepsilon\,|\kappa(u)|
	\right].
\end{equation}

\noindent\textbf{Case 3: $g > |\kappa(u)|$.}
Differentiating gives
\begin{equation}
	G'(g) = 2\left(\frac{\gamma}{\rho^2}
	+ \frac{\varepsilon}{2}\right)g
	- \gamma\,\frac{|\kappa(u)|}{\rho^2},
\end{equation}
so that $\partial_c G(g) = \{G'(g)\}$.

%%%%%%%%%%%%%%%%%%%%%%%%%%%%%%%%%%%%%%%%%%%%%%
\section{Proximal Bundle Algorithm}
\label{appendix:bundle}

We describe the proximal bundle method \cite{makela1992nonsmooth,makela2002survey} employed to compute the descent direction $d_k$ for the nonsmooth subproblem \eqref{eq:sub_u} . Let $f: \mathbb{R}^n \to \mathbb{R}$ be locally Lipschitz at the current iterate $x \in \mathbb{R}^n$. For an arbitrary subgradient $\xi \in \partial f(x)$, the \emph{$\xi$-linearization} of $f$ at $x$ is
\begin{equation}
	\bar{f}_\xi(y) := f(x) + \xi^\top(y - x), \quad y \in \mathbb{R}^n,
\end{equation}
and the \emph{linearization model} of $f$ at $x$ is
\begin{equation}
	\hat{f}_x(y) := \max\bigl\{\bar{f}_\xi(y) \mid \xi \in \partial f(x)\bigr\},
	\quad y \in \mathbb{R}^n.
\end{equation}
The following result underpins the direction-finding step of the algorithm.
\begin{lemma}
	Let $\mathcal{G} \subset \mathbb{R}^n$ be a convex compact set. For $p \in \mathcal{G}$,
	\begin{equation}
		p = \operatorname*{arg\,min}
		\bigl\{\|\xi\| \mid \xi \in \mathcal{G}\bigr\}
		\quad \Longleftrightarrow \quad
		p^\top \xi \geq \|p\|^2
		\quad \forall\, \xi \in \mathcal{G}.
	\end{equation}
\end{lemma}

\begin{theorem}
	Suppose $f: \mathbb{R}^n \to \mathbb{R}$ is locally Lipschitz at $x$ and let
	$\xi^* = \operatorname*{arg\,min}\bigl\{\|\xi\| \mid \xi \in \partial f(x)\bigr\}$.
	Consider the proximal subproblem
	\begin{equation}
		\min_{d \in \mathbb{R}^n}\; \hat{f}_x(x + d) + \tfrac{1}{2}\|d\|^2.
	\end{equation}
	Then this problem has a unique solution $d^* \in \mathbb{R}^n$ satisfying:
	\begin{enumerate}
		\item $-d^* = \xi^* \in \partial f(x)$,
		\item $f^\circ(x;\,d^*) = -\|d^*\|^2$,
		\item $\hat{f}_x(x + \lambda d^*) = \hat{f}_x(x)
		- \lambda\|\xi^*\|^2$ for all $\lambda \in [0,1]$,
		\item $0 \notin \partial f(x) \Longleftrightarrow d^* \neq 0$,
		\item $0 \in \partial f(x) \Longleftrightarrow \hat{f}_x$
		attains its global minimum at $x$.
	\end{enumerate}
\end{theorem}

Given the current iterate $x_k$ and auxiliary points
$\{y_j\}_{j \in J_k} \subset \mathbb{R}^n$ with associated subgradients
$\xi_j \in \partial f(y_j)$, the \emph{linearization error} and
\emph{subgradient locality measure} are defined, respectively, as
\begin{equation}
	\alpha_j^k := f(x_k) - f(y_j) - \xi_j^\top(x_k - y_j),
	\qquad
	\beta_j^k := \max\bigl\{|\alpha_j^k|,\;
	\gamma\|x_k - y_j\|^2\bigr\},
	\quad j \in J_k,
\end{equation}
where $\gamma \geq 0$ is a distance measure parameter. An approximation
$G_k(\varepsilon_k)$ of the $\varepsilon_k$-subdifferential
$\partial_{\varepsilon_k} f(x_k)$ is then constructed as
\begin{equation}
	G_k(\varepsilon_k) = \left\{
	\xi \in \mathbb{R}^n \;\Big|\;
	\xi = \sum_{j \in J_k} \lambda_j \xi_j,\;
	\sum_{j \in J_k} \lambda_j \beta_j^k \leq \varepsilon_k,\;
	\lambda_j \geq 0,\;
	\sum_{j \in J_k} \lambda_j = 1
	\right\},
\end{equation}
and the descent direction $d_k = -\xi^k$ is obtained by solving the quadratic program
\begin{equation}
	\label{eq:QP_direction}
	\begin{cases}
		\displaystyle\min_{\lambda}\;
		\dfrac{1}{2}\Bigl\|\sum_{j \in J_k} \lambda_j \xi_j\Bigr\|^2 \\[8pt]
		\text{subject to}\quad
		\displaystyle\sum_{j \in J_k} \lambda_j \beta_j^k \leq \varepsilon_k,\quad
		\displaystyle\sum_{j \in J_k} \lambda_j = 1,\quad
		\lambda_j \geq 0 \;\; \forall j \in J_k.
	\end{cases}
\end{equation}
The line search is governed by fixed parameters $m_L \in (0, \tfrac{1}{2})$, $m_R \in (m_L, 1)$, and $\bar{t} \in (0,1]$. At each iteration, we seek the largest $t_k^L \in [0,1]$ with $t_k^L \geq \bar{t}$ satisfying the \emph{sufficient decrease} condition
\begin{equation}
	\label{eq:armijo}
	f(x_k + t_k^L d_k) \leq f(x_k) + m_L\, t_k^L\, v_k,
	\qquad
	v_k = \hat{f}_k(x_k + d_k) - f(x_k) < 0.
\end{equation}
Three step types are distinguished:
\begin{itemize}
	\item \emph{Long serious step} ($t_k^L \geq \bar{t}$):
	$x_{k+1} = x_k + t_k^L d_k$, $y_{k+1} = x_{k+1}$.
	
	\item \emph{Short serious step} ($0 < t_k^L < \bar{t}$):
	$x_{k+1} = x_k + t_k^L d_k$, $y_{k+1} = x_k + t_k^R d_k$.
	
	\item \emph{Null step} ($t_k^L = 0$):
	$x_{k+1} = x_k$, $y_{k+1} = x_k + t_k^R d_k$,
\end{itemize}
where $t_k^R > t_k^L$ is chosen to satisfy the \emph{locality condition}
\begin{equation}
	\label{eq:locality}
	-\beta_{k+1}^{k+1} + \xi_{k+1}^\top d_k \geq m_R\, v_k.
\end{equation}
The complete procedure is summarized in Algorithm~\ref{alg:direction} and Algorithm~\ref{alg:PBA}.

\begin{algorithm}[h]
	\caption{Direction Computation}
	\label{alg:direction}
	\begin{algorithmic}[1]
		\STATE \textbf{Input:} Bundle index set $J_k$, tolerance $\varepsilon_k > 0$
		\WHILE{$\displaystyle\sum_{j \in J_k} \lambda_j \beta_j^k \geq \varepsilon_k$}
		\STATE Solve the quadratic program \eqref{eq:QP_direction}
		to obtain $\{\lambda_j\}_{j \in J_k}$
		\ENDWHILE
		\STATE Set $d_k = -\xi^k
		= -\displaystyle\sum_{j \in J_k} \lambda_j^k \xi_j$
		\STATE \textbf{Output:} Descent direction $d_k$
	\end{algorithmic}
\end{algorithm}

\begin{algorithm}[h]
	\caption{Proximal Bundle Algorithm}
	\label{alg:PBA}
	\begin{algorithmic}[1]
		\STATE \textbf{Data:} Optimality tolerance $\varepsilon_s > 0$;
		line search parameters $m_L \in (0,\tfrac{1}{2})$,
		$m_R \in (m_L, 1)$, $\bar{t} \in (0,1]$;
		distance measure $\gamma \geq 0$;
		bounds $x_m \leq x_M$.
		\STATE \textbf{Step 1 (Initialization).}
		Choose $x_1 \in \mathbb{R}^n$; set $y_1 := x_1$,
		$\mathcal{J}_1 := \{1\}$;
		compute $\xi_1 \in \partial f(y_1)$; set $k := 1$.
		\STATE \textbf{Step 2 (Direction finding).}
		Compute $d_k$ via Algorithm~\ref{alg:direction}.
		\STATE \textbf{Step 3 (Stopping criterion).}
		If $\|d_k\| + |v_k| < \varepsilon_s$, \textbf{stop}.
		\STATE \textbf{Step 4 (Line search).}
		Find $t_k^L \in [0,1]$ and $t_k^R \in [t_k^L, 1]$
		satisfying \eqref{eq:armijo} or \eqref{eq:locality}
		to determine the step type (null or serious).
		\STATE \textbf{Step 5 (Update).} Set
		\begin{align*}
			x_{k+1} &= \Pi_{[x_m,\,x_M]}\!\bigl(x_k + t_k^L d_k\bigr), \\
			y_{k+1} &= \Pi_{[x_m,\,x_M]}\!\bigl(x_k + t_k^R d_k\bigr).
		\end{align*}
		Compute $\xi_{k+1} \in \partial f(y_{k+1})$;
		choose $\mathcal{J}_{k+1} \subseteq \{1,\ldots,k+1\}$;
		set $k := k+1$ and go to Step~2.
	\end{algorithmic}
\end{algorithm}

%\bibliographystyle{cas-model2-names}	
%\bibliography{cas-refs}
%\clearpage

\end{document}